\documentclass[english]{amsart}
\usepackage{ae,aecompl}
\usepackage[T1]{fontenc}
\usepackage[latin9]{inputenc}
\usepackage{amssymb}

\makeatletter
\numberwithin{equation}{section} 
\numberwithin{figure}{section} 
  \theoremstyle{plain}
  \newtheorem{thm}{Theorem}[section]
  \theoremstyle{definition}
  \newtheorem{defn}[thm]{Definition}
  \theoremstyle{remark}
  \newtheorem*{rem*}{Remark}
 \theoremstyle{definition}
  \newtheorem{example}[thm]{Example}
  \theoremstyle{plain}
  \newtheorem{lem}[thm]{Lemma}
  \theoremstyle{remark}
  \newtheorem{rem}[thm]{Remark}
  \theoremstyle{plain}
  \newtheorem{cor}[thm]{Corollary}

\usepackage{babel}
\makeatother

\begin{document}

\title{Spectral conditions on Lie and Jordan algebras of compact operators}

\author{Matthew Kennedy}

\address{Department of Pure Mathematics \\
 University of Waterloo \\
 Waterloo, Ontario, Canada N2L 3G1}

\email{m3kennedy@uwaterloo.ca}

\author{Heydar Radjavi}

\address{Department of Pure Mathematics \\
 University of Waterloo \\
 Waterloo, Ontario, Canada N2L 3G1}

\email{hradjavi@uwaterloo.ca}

\begin{abstract}
We investigate the properties of bounded operators which satisfy a
certain spectral additivity condition, and use our results to study
Lie and Jordan algebras of compact operators. We prove that these
algebras have nontrivial invariant subspaces when their elements have
sublinear or submultiplicative spectrum, and when they satisfy simple
trace conditions. In certain cases we show that these conditions imply
that the algebra is (simultaneously) triangularizable.
\end{abstract}

\keywords{Invariant subspaces, spectrum, compact operators, Lie algebras, Jordan
algebras}

\subjclass[2000]{Primary 47A10, 47A15; Secondary 47L70}

\maketitle

\section{Introduction}

Conditions on the spectrum of an operator have been studied for some
time, for instance in the work of Motzkin and Taussky \cite{mot},
who investigated pairs of $n\times n$ matrices $A$ and $B$ with
the property that for every scalar $\lambda$, the eigenvalues of
$A+\lambda B$ are, subject to a slight technical condition, linear
functions of the eigenvalues of $A$ and $B$. Specifically, they
required that that the eigenvalues $\{\alpha_{1},...,\alpha_{n}\}$
and $\{\beta_{1},...,\beta_{n}\}$, of $A$ and $B$ respectively,
could be expressed as ordered sets $\{\alpha_{1},...,\alpha_{n}\}$
and $\{\beta_{1},...,\beta_{n}\}$, such that for every scalar $\lambda$,
the eigenvalues of $A+\lambda B$ are precisely $\alpha_{i}+\lambda\beta_{i}$,
for $i=1,...,n$. Such matrices are said to have property $L$ (``L''
is for ``linear''). 

A pair of (simultaneously) triangularizable matrices clearly has property
$L$, and Motzkin and Taussky were interested in conditions under
which the converse of this was true. They showed that a finite group
of matrices, with every pair of elements in the group having property
$L$, is not just triangularizable, but diagonalizable. Somewhat later,
Wales \cite{wal}, Zassenhaus \cite{zas}, and Guralnick \cite{gur}
showed that a multiplicative semigroup of $n\times n$ matrices is
triangularizable under the same circumstances.

The following conditions are seemingly much weaker than property $L$.

\begin{defn}
A pair of bounded operators $A$ and $B$ on a Banach space are said
to have \emph{subadditive spectrum} if\[
\sigma(A+B)=\sigma(A)+\sigma(B),\]
where $\sigma(A)+\sigma(B)$ means the set of all $\alpha+\beta$,
with $\alpha$ and $\beta$ in $\sigma(A)$ and $\sigma(B)$ respectively.
Similarly, $A$ and $B$ are said to have \emph{sublinear spectrum}
if \[
\sigma(A+\lambda B)\subseteq\sigma(A)+\lambda\sigma(B)\]
for every complex number $\lambda$. A family $\mathcal{F}$ of bounded
operators on a Banach space is said to have subadditive (resp. sublinear)
spectrum if every pair of elements in $\mathcal{F}$ has subadditive
(resp. sublinear) spectrum.
\end{defn}
\begin{rem*}
Note that for linear spaces, and in particular for Lie and Jordan
algebras, subadditivity of the spectrum clearly implies sublinearity.
\end{rem*}
The conditions of sublinear and subadditive spectrum are, in some
sense, even weaker than one might initially suspect. Indeed, as the
following example illustrates, there is little hope of obtaining any
results about the existence of invariant subspaces for an arbitrary
family of operators with subadditive or even sublinear spectrum without
imposing a great deal of additional structure.

\begin{example}
Consider the matrices\[
A=\left(\begin{array}{ccc}
0 & 1 & 0\\
0 & 0 & -1\\
0 & 0 & 0\end{array}\right),\quad B=\left(\begin{array}{ccc}
0 & 0 & 0\\
1 & 0 & 0\\
0 & 1 & 0\end{array}\right).\]
It is easy to verify that the linear space $\mathcal{S}$ spanned
by $A$ and $B$ consists entirely of nilpotent matrices, so $\mathcal{S}$
has sublinear spectrum, yet $\mathcal{S}$ clearly has no nontrivial
invariant subspaces.
\end{example}
It was shown in \cite{rad3} that the property of sublinear spectrum
implies the triangularizability of a semigroup of compact operators,
which is an extension of the classical results mentioned above to
the context of an infinite-dimensional Banach space.

A natural question is whether this kind of result holds for other
types of algebraic structures. In this paper we study Lie and Jordan
algebras of compact operators satisfying spectral conditions like
sublinearity and submultiplicativity, which have been shown to imply
the existence of invariant subspaces for semigroups of compact operators.
We will prove that many of the results obtained for semigroups also
hold in this non-associative context.

Shulman and Turovskii obtained several results of this nature in their
development of a radical theory for Lie algebras of compact operators
\cite{shu2}. They showed that a Lie algebra of compact operators
with subadditive spectrum has invariant subspaces \cite{shu2}, and
it was implicit in their proof that this condition actually implied
triangularizability. This gave an extension of Engel's Theorem for
Lie algebras of compact operators, which was also proven recently
in \cite{shu1}.

On the other hand, results on the existence of invariant subspaces
for Jordan algebras of compact operators have only recently been obtained.
In particular, there is currently no Jordan analogue of the theory
which was developed in \cite{shu2}. This is one reason why our results
have been developed in a different manner.

We begin by analyzing the properties of bounded linear operators which
satisfy a certain spectral additivity condition. The results we obtain
here, when combined with new ``Cartan-like'' conditions for the existence
of invariant subspaces in Lie and Jordan algebras of compact operators,
provides us with a general method of proving the existence of invariant
subspaces for these algebras, in particular when their elements have
subadditive or submultiplicative spectrum. We will show that this
implies a Lie or Jordan algebra of compact operators with subadditive
spectrum is triangularizable, obtaining a new proof of Shulman and
Turovskii's result for Lie algebras. Finally, we obtain simple conditions
on the trace of the finite-rank elements in the algebra which imply
the existence of invariant subspaces. As far as we know, some of our
results are new even in finite dimensions.

\section{Preliminaries}

In this paper we confine ourselves to the field $\mathbb{C}$ of complex
numbers. For a Banach space $\mathcal{X}$, we let $\mathcal{B}(\mathcal{X})$
and $\mathcal{K}(\mathcal{X})$ denote the algebras of all bounded
and compact operators on $\mathcal{X}$ respectively.

For $A$ in $\mathcal{B}(\mathcal{X})$, we let $\sigma(A)$ and $r(A)$
denote the spectrum and spectral radius of $A$ respectively. If $A$
is of finite rank, then the trace of $A$ is well-defined; we denote
it by $\operatorname{tr}(A)$ .

To clarify our exposition, we will sometimes make use of shorthand
notation. For example, if $\mathcal{S}_{1}$ and $\mathcal{S}_{2}$
are subsets of $\mathcal{B}(\mathcal{X})$, then we will write $\mathcal{S}_{1}^{n}$
for the linear span of $\{A^{n}:A\in\mathcal{S}_{1}$\}, and $\mathcal{S}_{1}\mathcal{S}_{2}$
for the linear span of $\{AB:A\in\mathcal{S}_{1},B\in\mathcal{S}_{2}\}$.

A family $\mathcal{F}$ of operators is called reducible if there
is a nontrivial closed subspace invariant under every member of $\mathcal{F}$.
We say that $\mathcal{F}$ is triangularizable if its lattice of invariant
subspaces contains a maximal subspace chain $\mathcal{C}$. (If $\mathcal{M}_{1}$
and $\mathcal{M}_{2}$ are two members of $\mathcal{C}$ with $\mathcal{M}_{1}\subseteq\mathcal{M}_{2}$
and no other member between $\mathcal{M}_{1}$ and $\mathcal{M}_{2}$,
then $\operatorname{dim}(\mathcal{M}_{1}\ominus\mathcal{M}_{2})\leq1$.) 

To establish the triangularizability of a family of operators, we
will require the following lemma from \cite[Lemma 7.1.11]{rad1}.

\begin{lem}
[The Triangularization Lemma]Let $\mathcal{P}$ be a property of
a family of operators such that
\begin{enumerate}
\item every family of operators with property $\mathcal{P}$ is reducible,
and
\item \label{enu:inher-by-quot}if $\mathcal{F}$ has property $\mathcal{P}$
and if $\mathcal{M}_{1}$ and $\mathcal{M}_{2}$ are invariant subspaces
of $\mathcal{F}$ with $\mathcal{M}_{1}\subseteq\mathcal{M}_{2}$,
then $\hat{\mathcal{F}}$ has property $\mathcal{P}$, where $\hat{\mathcal{F}}$
is the set of all quotient operators on $\mathcal{M}_{1}/\mathcal{M}_{2}$
induced by $\mathcal{F}$. 
\end{enumerate}
Then every family of operators with property $\mathcal{P}$ is triangularizable.
\end{lem}
If $\mathcal{P}$ is a property of a family of operators which satisfies
hypothesis \ref{enu:inher-by-quot} of the Triangularization Lemma,
then we say that $\mathcal{P}$ is inherited by quotients.

The following result from \textbf{\cite[Corollary 8.4.2]{rad1}} will
be used in combination with the Triangularization Lemma.

\begin{lem}
\label{lem:sub-spr-quo} The property of sublinear spectrum is inherited
by quotients.
\end{lem}
\begin{defn}
A pair of bounded operators $A$ and $B$ on a Banach space are said
to have \emph{submultiplicative spectrum} if\[
\sigma(AB)\subseteq\sigma(A)\sigma(B),\]
where $\sigma(A)\sigma(B)$ means the set of all $\alpha\beta$, with
$\alpha$ and $\beta$ in $\sigma(A)$ and $\sigma(B)$ respectively.
A family $\mathcal{F}$ of bounded operators on a Banach space are
said to have submultiplicative spectrum if every pair of elements
in $\mathcal{F}$ has submultiplicative spectrum.
\end{defn}
The property of submultiplicative specrum is not strong enough to
imply even the reducibility of a semigroup of compact operators; indeed,
the existence of finite irreducible matrix groups with sublinear specrum
was shown in \cite{rad2}. We will be able to show however, that Lie
and Jordan algebras of compact operators with submultiplicative spectrum
are reducible.

\section{\label{sec:rigidity}Operators with Stable Spectrum}

It turns out that the following property, obviously closely related
to the property of sublinear spectrum, is of particular importance
for obtaining many of our reducibility results.

\begin{defn}
Let $T$ be a bounded operator on a Banach space. We say that a bounded
operator $A$ has $T$\emph{-stable spectrum} if\[
r(A+\lambda T)\leq r(A)\]
for every complex number $\lambda$. A family of bounded operators
on a Banach space is said to have $T$-stable spectrum if each of
its elements has $T$-stable spectrum.
\end{defn}
\begin{rem*}
It will sometimes be useful to reference a family of bounded operators
with $T$-stable spectrum without making explicit mention of $T$.
Therefore, we will say that a family of bounded operators has stable
spectrum if it has $T$-stable spectrum, for some nonzero $T$.
\end{rem*}
\begin{example}
\label{exa:pre-example}Consider the matrices \[
A=\left(\begin{array}{ccc}
-1 & 0 & 0\\
0 & 0 & 1\\
0 & 0 & 0\end{array}\right),\quad B=\left(\begin{array}{ccc}
0 & -1 & -1\\
0 & -1 & -1\\
1 & 0 & 1\end{array}\right).\]
Note that $A$ is in Jordan normal form, so $\sigma(A)=\{-1,0\}$,
and $B^{3}=0$, so $\sigma(B)=\{0\}$. The characteristic polynomial
of $A+\lambda B$ is $t^{3}-t^{2}$, which implies that $\sigma(A+\lambda B)=\sigma(A)$
for all $\lambda$in $\mathbb{C}$. This shows that $A$ is $B$-stable.

Calculating the inverse of $\mu-A$, we have \[
(\mu-A)^{-1}B=\frac{1}{\mu^{2}(1-\mu)}\left(\begin{array}{ccc}
0 & \mu^{2} & \mu^{2}\\
-1 & -\mu(1-\mu) & -(1-\mu)^{2}\\
\mu(1-\mu) & 0 & \mu(1-\mu)\end{array}\right),\]
and it is routine to verify that this matrix is nilpotent. Calculating
the inverse of $1-\mu B$, we have\[
(1-\mu B)^{-1}A=\left(\begin{array}{ccc}
0 & 0 & 0\\
-1 & 0 & -1\\
2 & 0 & 1\end{array}\right),\]
and this matrix has characteristic polynomial $t^{3}-t^{2}$, showing
that it has the same eigenvalues as $A$.
\end{example}
In this section, we will show that the results in Example \ref{exa:pre-example}
hold in general, which will be important for our main results. An
important tool will be the theory of subharmonic functions, based
on the result of Vesentini that if $f$ is an analytic function from
a domain of the complex numbers into a Banach algebra, then the functions
$\lambda\rightarrow r(f(\lambda))$ and $\lambda\rightarrow\operatorname{log}(r(f(\lambda)))$
are subharmonic \cite{ves}. We will require the following two fundamental
results from the theory of subharmonic functions (see for example
\cite[Theorem A.1.3]{aup} and \cite[Theorem A.1.29]{aup} resp.).

\begin{thm}
[Maximum Principle for Subharmonic Functions]Let $f$ be a subharmonic
function on a domain $D$ of $\mathbb{C}$. If there exists $\lambda_{0}$
in $D$ such that $f(\lambda)\le f(\lambda_{0})$ for all $\lambda$
in $D$, then $f(\lambda)=f(\lambda_{0})$ for all $\lambda$ in $D$.
\end{thm}
We state here only a special case of H. Cartan's Theorem.

\begin{thm}
[H. Cartan's Theorem]Let $f$ be a subharmonic function from a domain
$D$ of $\mathbb{C}$. If $f(\lambda)=-\infty$ on a nonempty open
open ball in $D$, then $f(\lambda)=-\infty$ for all $\lambda$ in
$D$.
\end{thm}
\begin{rem}
\label{rem:const-sp-rad}For bounded operators $A$ and $T$ on a
Banach space, the function $\lambda\to A+\lambda T$ is analytic,
so by Vesentini's results, the functions $\lambda\to r(A+\lambda T)$
and $\lambda\to\operatorname{log}(r(A+\lambda T))$ are subharmonic.
If $A$ has $T$-stable spectrum, the Maximum Principle for subharmonic
functions immediately implies that $r(A+\lambda T)=r(A)$ for all
complex numbers$\lambda$.

If $A$ and $T$ have sublinear spectrum and $T$ is quasinilpotent,
then $A$ has $T$-stable spectrum. The following lemma from \cite[Lemma 4.2]{shu2})
shows that the quasinilpotence of $T$ is a necessary condition for
$A$ to be $T$-stable. We provide the proof here for the convenience
of the reader, and as a warmup for further applications of subharmonicity.
\end{rem}
\begin{lem}
\label{lem:stable-quasinilpotence}Let $A$ and $T$ be bounded operators
on a Banach space. If $A$ has $T$-stable spectrum, then $T$ is
quasinilpotent. 
\end{lem}
\begin{proof}
By Remark \ref{rem:const-sp-rad}, $r(A+\lambda T)=r(A)$ for all
$\lambda$in $\mathbb{C}$, so \[
r(\lambda^{-1}A+T)=|\lambda|^{-1}r(A)\]
for all nonzero $\lambda$ in $\mathbb{C}$. Thus, by the subharmonicity
of $r(\lambda^{-1}A+T)$,\[
r(T)=\lim\sup_{\lambda\rightarrow\infty}r(\lambda^{-1}A+T)=0.\]

\end{proof}
\begin{lem}
\label{lem:stable-inverse-quasinilpotent}Let $A$ and $T$ be bounded
operators on a Banach space. Then $A$ has $T$-stable spectrum if
and only if $(\mu-A)^{-1}T$ is quasinilpotent for all $\mu\notin\sigma(A)$. 
\end{lem}
\begin{proof}
By remark \ref{rem:const-sp-rad}, $r(A+\lambda T)=r(A)$ for all
$\lambda$ in $\mathbb{C}$, so for $\mu$ in $\mathbb{C}$ with $|\mu|>r(A)$,
both $\mu-A$ and $\mu-A-\lambda T$ are invertible. Therefore,\[
\lambda^{-1}(\mu-A)^{-1}(\mu-A-\lambda T)=\lambda^{-1}-(\mu-A)^{-1}T\]
is invertible for all nonzero $\lambda$ in $\mathbb{C}$. This means
that the values of the operator-valued function $\mu\rightarrow(\mu-A)^{-1}T$,
which is analytic for $\mu\notin\sigma(A)$, are quasinilpotent whenever
$|\mu|>r(A)$.

Consider the subharmonic function $\mu\rightarrow\operatorname{log}(r((\mu-A)^{-1}T))$
defined for $\mu\notin\sigma(A)$. Since $\operatorname{log}(r((\mu-A)^{-1}T))=-\infty$
whenever $|\mu|>r(A)$, by H. Cartan's Theorem, $\operatorname{log}(r((\mu-A)^{-1}T))=-\infty$
for all $\mu\notin\sigma(A)$. In other words, $(\mu-A)^{-1}T$ is
quasinilpotent for all $\mu\notin\sigma(A)$.
\end{proof}
\begin{lem}
\label{lem:stable-spec-containment}Let $A$ and $T$ be bounded operators
on a Banach space. Then $A$ has $T$-stable spectrum if and only
if $\sigma(A+\lambda T)\subseteq\sigma(A)$ for all $\lambda$ in
$\mathbb{C}$.
\end{lem}
\begin{proof}
Suppose $\mu\in\sigma(A+\lambda T)$, but that $\mu\notin\sigma(A)$
. Then clearly $\lambda$ is nonzero, and\[
\lambda^{-1}(\mu-A)^{-1}(\mu-A-\lambda T)=\lambda^{-1}-(\mu-A)^{-1}T\]
is not invertible. But by Theorem \ref{lem:stable-inverse-quasinilpotent},
$(\mu-A)^{-1}T$ is quasinilpotent for all $\mu\notin\sigma(A)$,
which gives a contradiction. 
\end{proof}
The next result also follows from \cite[Theorem 3.4.14]{aup}, but
it is interesting to see that it can be proved in the following way.

\begin{lem}
\label{lem:stable-const-spec}Let $A$ and $T$ be bounded operators
on a Banach space. If $A$ has $T$-stable spectrum and $\sigma(A)$
has no interior points, then $\sigma(A+\lambda T)=\sigma(A)$ for
all $\lambda$ in $\mathbb{C}$.
\end{lem}
\begin{proof}
This follows immediately from Corollary \ref{lem:stable-spec-containment}
and the Spectral Maximum Principle of \cite{aup}. 
\end{proof}
\begin{lem}
\label{lem:stable-inverse-const}Let $A$ and $T$ be bounded operators
on a Banach space. If $A$ has $T$-stable spectrum and $\sigma(A)$
has no interior points, then $\sigma((1-\nu T)^{-1}A)=\sigma(A)$
for all $\nu$ in $\mathbb{C}$. 
\end{lem}
\begin{proof}
First suppose $\lambda$ is nonzero, and that $\lambda\notin\sigma(A)$.
By Lemma \ref{lem:stable-const-spec}, $\sigma(\lambda^{-1}A+\nu T)=\sigma(\lambda^{-1}A)$,
and by Lemma \ref{lem:stable-quasinilpotence}, $T$ is quasinilpotent.
These two facts imply that $1-\nu T$ and $1-\lambda^{-1}A-\nu T$
are both invertible, and hence that \[
\lambda(1-\nu T)^{-1}(1-\lambda^{-1}A-\nu T)=\lambda-(1-\nu T)^{-1}A\]
 is invertible for all $\nu$ in $\mathbb{C}$. Therefore, $\lambda\notin\sigma((1-\nu T)^{-1}A)$
for all $\nu$ in $\mathbb{C}$.

Now suppose $0\notin\sigma(A)$. Then $A$ is invertible, implying
$(1-\nu T)^{-1}A$ is invertible, and hence by the quasinilpotence
of $T$, that $0\notin\sigma((1-\nu T)^{-1}A)$ for all $\nu$ in
$\mathbb{C}$.

We have shown that $\sigma((1-\nu T)^{-1}A)\subseteq\sigma(A)$ for
all $\nu$ in $\mathbb{C}$. Since $\sigma(A)$ has no interior points,
the result now follows from the Spectral Maximum Principle of \cite{aup}.
\end{proof}
\begin{lem}
\label{lem:bounded-rank-const-trace}Let $f$ be an entire function
from $\mathbb{C}$ into $\mathcal{B}(\mathcal{X})$. Suppose that
\begin{enumerate}
\item there exists a complex number $\lambda_{0}$ such that $\sigma(f(\lambda_{0}))=\sigma(f(\lambda))$,
and that
\item there exists $N$ such that $\operatorname{rank}(f(\lambda)\leq N$
for all complex numbers $\lambda$.
\end{enumerate}
Then $\operatorname{tr}(f(\lambda))=\operatorname{tr}(f(\lambda_{0}))$
for all $\lambda$ in $\mathbb{C}$.
\end{lem}
\begin{proof}
Since $f$ takes finite-rank values and is entire, the function $\lambda\to\operatorname{tr}(f(\lambda))$
is also entire. For each $\lambda$, the trace of $f(\lambda)$ is
the sum, with multiplicity, of the eigenvalues of $f(\lambda)$, so
we may write it as a finite sum\[
\operatorname{tr}(f(\lambda))=\sum_{\alpha\in\sigma(f(\lambda_{0}))}n_{\alpha}(\lambda)\alpha,\]
where $n_{\alpha}(\lambda)$ denotes the multiplicity of the eigenvalue
$\alpha$ with respect to $f(\lambda)$. But clearly\[
\sum_{\alpha\in\sigma(f(\lambda_{0}))}n_{\alpha}(\lambda)|\alpha|\le\operatorname{rank}(f(\lambda))\|f(\lambda_{0})\|\leq N\|f(\lambda_{0})\|,\]
which implies that the function $\lambda\to\operatorname{tr}(f(\lambda))$
is bounded. By Liouville's Theorem, it now follows that $\operatorname{tr}(f(\lambda))=\operatorname{tr}(f(\lambda_{0}))$
for all $\lambda$ in $\mathbb{C}$. 
\end{proof}
The next result extends \cite[Lemma 4.2]{shu2}, and gives a symmetric
trace condition which will be useful for our results.

\begin{lem}
\label{lem:rig-zer-tr}Let $A$ and $B$ be bounded operators on a
Banach space. If $A$ is $B$-stable and one of $A$ or $B$ is of
finite rank, then $\operatorname{tr}(A^{n}B)=\operatorname{tr}(AB^{n})=0$
for all $n\geq1$.
\end{lem}
\begin{proof}
First suppose that $A$ is of finite rank. Since $B$ is quasinilpotent
by Lemma \ref{lem:stable-quasinilpotence}, the function $\nu\to(1-\nu B)^{-1}A$
is entire. Moreover, $\sigma((1-\nu B)^{-1}A)=\sigma(A)$ for all
$\nu$ in $\mathbb{C}$ by Lemma \ref{lem:stable-inverse-const}.
Then, taking $n$-th powers, the function $v\to((1-\nu B)^{-1}A)^{n}$
is also entire, and $\sigma(((1-\nu B)^{-1}A)^{n})=\sigma(A^{n})$
for all $\nu$ in $\mathbb{C}$. Clearly $\operatorname{rank}(((1-\nu B)^{-1}A)^{n})\leq\operatorname{rank}(A)$,
so $\operatorname{tr}(((1-\nu B)^{-1}A)^{n})=\operatorname{tr}(A^{n})$
for all $\nu$ in $\mathbb{C}$ by Lemma \ref{lem:bounded-rank-const-trace}.

For $|\nu|<\|B\|^{-1}$, we may expand $(1-\nu B)^{-1}A$ as a power
series in $\nu$,\[
(1-\nu B)^{-1}A=\sum_{k\geq0}B^{k}A\nu^{k}.\]
Hence\[
((1-\nu B)^{-1}A)^{n}=(\sum_{k\geq0}B^{k}A\nu^{k})^{n}.\]
The coefficient of $\nu^{k}$ in the above expansion is $B^{k}A$,
and for $n\ge1$, the coefficient of $\nu$ is\[
BA^{n}+ABA^{n-1}+...+A^{n-1}BA.\]
But we may also expand the constant function $\operatorname{tr}((1-\nu B)^{-1}A)^{n})$
as a power series in $\nu$, and the linearity of the trace implies
that for $n=1$, the coefficient of $\nu^{k}$ in this expansion is
$\operatorname{tr}(B^{k}A)$, and for $n\geq1$, that the coefficient
of $\nu$ is\[
\operatorname{tr}(BA^{n}+ABA^{n-1}+...+A^{n-1}BA)=n\operatorname{tr}(A^{n}B).\]
Comparing the coefficients on the left and right hand side of the
equation $\operatorname{tr}(((1-\nu B)^{-1}A)^{n})=\operatorname{tr}(A^{n})$
therefore gives $\operatorname{tr}(A^{n}B)=0$ for all $n\geq1$,
and $\operatorname{AB^{k}}=0$ for all $k\geq1$. 

Now suppose that $B$ is of finite rank. The function $\nu\rightarrow(1-\nu A)^{-1}B$
is analytic for $\nu^{-1}\not\in\sigma(A)$, with quasinilpotent values
by Lemma \ref{lem:stable-inverse-quasinilpotent}. Taking $n$-th
powers, the function $\nu\to((1-\nu A)^{-1}B)^{n}$ is also analytic
with quasinilpotent values for $\nu^{-1}\not\in\sigma(A)$. This means
$\operatorname{tr}(((1-\nu A)^{-1}B)^{n})=0$ for all $\nu\notin\sigma(A)$.

As above, for $|\nu|<\|A\|$, we may expand $((1-\nu A)^{-1}B)^{n}$
as a power series in $\nu$, \[
((1-\nu A)^{-1}B)^{n}=(\sum_{k\geq0}A^{k}B\nu^{k})^{n}.\]
For $n=1$, the coefficient of $\nu^{k}$ in the above expansion is
$A^{k}B$, and for $n\geq1$, the coefficient of $\nu$ is\[
AB^{n}+BAB^{n-1}+...+B^{n-1}A.\]
Proceeding as before, we may also expand the constant function $\operatorname{tr}(((1-\nu A)^{-1}B)^{n})$
as a power series in $\nu$, and the linearity of the trace implies
that for $n=1$, the coefficient of $\nu^{k}$ in this expansion is
$\operatorname{tr}(A^{k}B)$, and for $n\geq1$, that the coefficient
of $\nu$ is\[
\operatorname{tr}(AB^{n}+BAB^{n-1}+...+B^{n-1}AB)=n\operatorname{tr}(AB^{n}).\]
Comparing the coefficients of the left and right hand side of the
equation $\operatorname{tr}(((1-\nu A)^{-1}B)^{n})=0$ therefore gives
$\operatorname{tr}(A^{k}B)=0$ for all $k\geq1$, and $\operatorname{tr}(AB^{n})=0$
for all $n\geq1$. 
\end{proof}

\section{Spectral Conditions on a Lie Algebra of Compact Operators}

In this section we study Lie algebras of compact operators which satisfy
spectral properties like sublinearity and submultiplicativity.

Recall that an operator Lie algebra $\mathcal{L}$ is a subspace of
$\mathcal{B}(\mathcal{X})$ which is closed under the Lie commutator
product $[A,B]=AB-BA$, for $A,B\in\mathcal{L}$. A Lie ideal $\mathcal{I}$
of $\mathcal{L}$ is a Lie subalgebra of $\mathcal{L}$ such that
$[A,B]\in\mathcal{I}$ whenever $A\in\mathcal{L}$ and $B\in\mathcal{I}$.
For $A$ in $\mathcal{L}$, we define the bounded linear operator
$\operatorname{ad}_{A}$ on $\mathcal{L}$ by $\operatorname{ad}_{A}(B)=[A,B].$

\begin{defn}
An operator Lie algebra $\mathcal{L}$ is said to be an \emph{Engel
Lie algebra} if $\operatorname{ad}_{A}$ is quasinilpotent for every
$A$ in $\mathcal{L}$. An ideal of $\mathcal{L}$ is said to be an
\emph{Engel ideal} if it is an Engel Lie algebra.
\end{defn}
\begin{rem*}
There are two particularly important classes of Engel Lie algebras:
\begin{enumerate}
\item Commutative operator Lie algebras are Engel Lie algebras, since $\operatorname{ad}_{A}=0$
for every $A\in\mathcal{L}$.
\item Operator Lie algebras in which every element is compact and quasinilpotent
are Engel Lie algebras. Indeed, for $A\in\mathcal{L}$, $\operatorname{ad}_{A}$
is the restriction of $L_{A}-R_{A}$ to $\mathcal{L}$, where $L_{A}$
and $R_{A}$ denote the (commuting) operators of left and right multiplication
by $A$ on $\mathcal{B}(\mathcal{X})$ respectively. We always have
$\sigma(L_{A})\subseteq\sigma(A)$ and $\sigma(R_{A})\subseteq\sigma(A)$,
so $L_{A}$ and $R_{A}$ are quasinilpotent. It follows that $L_{A}-R_{A}$
is quasinilpotent, and hence so is its restriction to $\mathcal{L}$.
\end{enumerate}
\end{rem*}
We require the following two important results of Shulman and Turovskii.
The first result from \cite[Corollary 11.6]{shu1} is an extension
of Engel's well-known theorem to Lie algebras of compact operators.
The second result from \cite[Corollary 4.20]{shu2} establishes the
abundance of finite rank operators in an irreducible Lie algebra of
compact operators.

\begin{thm}
\label{thm:lie-engel-ideal}A Lie algebra of compact operators which
contains a nonzero Engel Lie ideal is reducible. In particular, an
Engel Lie algebra of compact operators is triangularizable.
\end{thm}

\begin{thm}
\label{thm:lie-no-finite-ranks}A uniformly closed Lie algebra of
compact operators which doesn't contain any nilpotent finite-rank
elements is triangularizable.
\end{thm}
For a general Lie algebra of compact operators, it is often more tractable
to study its ideal of finite-rank operators due to such niceties as
the existence of a trace, so we are especially interested in situations
when the reducibility of this ideal implies the reducibility of the
entire Lie algebra. Shulman and Turovskii raised the question \cite{shu2}
of whether a Lie algebra of compact operators is reducible whenever
its ideal of finite-rank operators is reducible. That such a result
holds for an associative algebra of compact operators is a consequence
of Lomonosov's Theorem (see for example \cite[Theorem 7.4.7]{rad1}).
The next result, a generalization of \cite[Lemma 2.3]{ken2}, provides
at least one example of a situation in which this type of result is
true in the Lie algebra case.

\begin{thm}
\label{thm:lie-zero-trace-ideal}Let $\mathcal{L}$ be a uniformly
closed Lie algebra of compact operators. If $\mathcal{L}$ contains
a nonzero element $A$, with the property that $\operatorname{tr}(AF)=0$
for every finite-rank operator $F$ in $\mathcal{L}$, then $\mathcal{L}$
is reducible.
\end{thm}
\begin{proof}
Let $\mathcal{F}$ denote the ideal of finite operators in $\mathcal{L}$,
and let $\mathcal{I}=\{A\in\mathcal{L}:\operatorname{tr}(AF)=0\mbox{ for all }F\in\mathcal{F}\}$.
Using the identity\[
\operatorname{tr}(\{A,B\}F)=\operatorname{tr}(A\{B,F\}),\]
it is easy to verify that $\mathcal{I}$ is a Lie ideal of $\mathcal{L}$.
Consider the Lie ideal $\mathcal{I}_{\mathcal{F}}=\mathcal{I}\cap\mathcal{F}$.
If $\mathcal{I}_{\mathcal{F}}=0$, then $\mathcal{I}$ is triangularizable
by Theorem \ref{thm:lie-no-finite-ranks}, and thus $\mathcal{L}$
is reducible by Theorem \ref{thm:lie-engel-ideal}. Hence we may suppose
that $\mathcal{I}_{\mathcal{F}}\ne0$.

We have $\operatorname{tr}(AB)=0$ for all $A,B\in\mathcal{I}_{\mathcal{F}}$,
which implies by \cite[Theorem 4.5]{ken1} that $[\mathcal{I}_{\mathcal{F}},\mathcal{I}_{\mathcal{F}}]$
consists of nilpotent elements. Note that $[\mathcal{I}_{\mathcal{F}},\mathcal{I}_{\mathcal{F}}]$
is a Lie ideal of $\mathcal{L}$. Indeed, for $A\in\mathcal{L}$ and
$F,G\in\mathcal{I}_{\mathcal{F}}$,\[
[A,[F,G]]=-[F,[G,A]]-[G,[A,F]]\]
by the Jacobi identity. Hence $\mathcal{I}_{\mathcal{F}}$ is triangularizable
by \ref{thm:lie-engel-ideal}. If $[\mathcal{I}_{\mathcal{F}},\mathcal{I}_{\mathcal{F}}]=0$,
then $\mathcal{I}_{\mathcal{F}}$ is a commutative Lie ideal, and
hence is triangularizable by Theorem \ref{thm:lie-engel-ideal}. Otherwise,
the triangularizable ideal $[\mathcal{I}_{\mathcal{F}},\mathcal{I}_{\mathcal{F}}]$
is nonzero. Either way, $\mathcal{L}$ has a nonzero triangularizable
ideal, which implies by Theorem \ref{thm:lie-engel-ideal} that it
is reducible.
\end{proof}
\begin{lem}
\label{thm:lie-rig-red}Let $\mathcal{L}$ be a Lie algebra of compact
operators with a nonzero element $T$. If $\mathcal{L}$ has $T$-stable
spectrum, then $\mathcal{L}$ is reducible.
\end{lem}
\begin{proof}
By the continuity of the spectrum of compact operators, we may suppose
that $\mathcal{L}$ is uniformly closed. Let $\mathcal{F}$ be the
Lie ideal of finite-rank operators in $\mathcal{L}$. By Theorem \ref{lem:rig-zer-tr},
$\operatorname{tr}(FT)=0$ for all $F$ in $\mathcal{F}$, so the
result follows by Theorem \ref{thm:lie-zero-trace-ideal}.
\end{proof}
\begin{thm}
\label{thm:lie-sub-tri}A Lie algebra of compact operators with subadditive
spectrum is triangularizable.
\end{thm}
\begin{proof}
Let $\mathcal{L}$ be a Lie algebra of compact operators with sublinear
spectrum. By the Triangularization Lemma and Lemma \ref{lem:sub-spr-quo},
it suffices to show the reducibility of $\mathcal{L}$. As in the
proof of Lemma \ref{thm:lie-rig-red}, we may suppose that $\mathcal{L}$
is uniformly closed. Then by Theorem \ref{thm:lie-no-finite-ranks},
if $\mathcal{L}$ doesn't contain a nonzero finite-rank nilpotent
element, then it is reducible; hence we may suppose that some nonzero
$T$ in $\mathcal{L}$ is nilpotent, and of finite rank. Then by the
hypothesis of sublinear spectrum, $\mathcal{L}$ has $T$-stable spectrum,
so the result follows by Lemma \ref{thm:lie-rig-red}.
\end{proof}
\begin{thm}
A Lie algebra of compact operators with submultiplicative spectrum
is reducible.
\end{thm}
\begin{proof}
Let $\mathcal{L}$ be a Lie algebra of compact operators with submultiplicative
spectrum. As in the proof of Lemma \ref{thm:lie-rig-red}, we may
suppose that $\mathcal{L}$ is uniformly closed. Let $\mathcal{F}$
be the ideal of finite-rank elements in $\mathcal{L}$. By Theorem
\ref{thm:lie-no-finite-ranks}, if $\mathcal{L}$ doesn't contain
a nonzero finite-rank nilpotent element, then it is reducible; hence
we may suppose that some nonzero $T$ in $\mathcal{L}$ is nilpotent,
and of finite rank. Then for every $F$ in $\mathcal{F}$, $\sigma(FT)=0$
by the hypothesis of submultiplicative spectrum, so $\operatorname{tr}(FT)=0$.
Hence $\mathcal{L}$ is reducible by Theorem \ref{thm:lie-zero-trace-ideal}. 
\end{proof}
\begin{thm}
Let $\mathcal{L}$ be a uniformly closed Lie algebra of compact operators
with the property that $\operatorname{tr}(FGH)=\operatorname{tr}(FHG)$
whenever $F,G,H\in\mathcal{L}$ are of finite rank. Then $\mathcal{L}$
is reducible.
\end{thm}
\begin{proof}
Let $\mathcal{F}$ be the ideal of finite-rank operators in $\mathcal{L}$.
If $\mathcal{F}$ is commutative, then it is an Engel ideal of $\mathcal{L}$,
and $\mathcal{L}$ is reducible by Theorem \ref{thm:lie-engel-ideal}.
Hence we may suppose that $[G,H]\ne0$ for some $G,H\in\mathcal{L}$.
Then for all $F$ in $\mathcal{F}$, $\operatorname{tr}(F[G,H])=0$,
so $\mathcal{L}$ is reducible by Theorem \ref{thm:lie-zero-trace-ideal}.
\end{proof}
\begin{thm}
Let $\mathcal{L}$ be a uniformly closed Lie algebra of compact operators
with the property that $\operatorname{tr}(F^{2})=0$ for every finite-rank
element $F$ in $\mathcal{L}$. Then $\mathcal{L}$ is reducible. 
\end{thm}
\begin{proof}
Let $\mathcal{F}$ be the ideal of finite-rank operators in $\mathcal{L}$.
Then for any $F,G\in\mathcal{F}$,\begin{eqnarray*}
0 & = & \operatorname{tr}((F+G)^{2})\\
 & = & \operatorname{tr}(F^{2})+\operatorname{tr}(G^{2})+\operatorname{tr}(FG+GF)\\
 & = & \operatorname{tr}(FG+GF)\\
 & = & 2\operatorname{tr}(FG),\end{eqnarray*}
so $\operatorname{tr}(FG)=0$. If $\mathcal{F}=0$, then $\mathcal{L}$
is reducible by Theorem \ref{thm:lie-no-finite-ranks}. Otherwise,
$\mathcal{L}$ is reducible by Theorem \ref{thm:lie-zero-trace-ideal}.
\end{proof}

\section{Spectral Conditions on a Jordan Algebra of Compact Operators}

In this section we turn our attention to Jordan algebras, and show
that most of the results we obtained for Lie algebras are also true
in this setting.

An operator Jordan algebra $\mathcal{J}$ is a subspace of $\mathcal{B}(\mathcal{X})$
which is closed under the Jordan anticommutator product $\{A,B\}=AB-BA$,
for $A,B\in\mathcal{J}$. It is easy to verify that this is equivalent
to $\mathcal{J}$ being closed under taking positive integer powers.
A Jordan ideal $\mathcal{I}$ of $\mathcal{J}$ is a Jordan subalgebra
of $\mathcal{J}$ such that $\{A,B\}\in\mathcal{I}$ whenever $A\in\mathcal{J}$
and $B\in\mathcal{I}$.

The methods of this section will differ from those used in the section
on Lie algebras. This is mainly because for obtaining reducibility
results, the closest Jordan analogue of an Engel Lie ideal will be
a Jordan ideal in which every element is quasinilpotent. This is due
to the following result, the Jordan analogue of Theorem \ref{thm:lie-engel-ideal},
which was recently obtained in \textbf{\cite[Theorem 11.3]{ken2}.}

\begin{thm}
\textbf{\label{thm:jor-volterra-ideal}}A Jordan algebra of compact
operators with a nonzero ideal of quasinilpotent operators is reducible.
In particular, a Jordan algebra of compact quasinilpotent operators
is triangularizable.
\end{thm}
Let $\mathcal{J}$ be a uniformly closed Jordan algebra of compact
operators, and let $A$ be a non-quasinilpotent element of $\mathcal{J}$.
For nonzero $\lambda$ in $\sigma(A)$, it is well known that the
Riesz projection $P_{\lambda}$ of $A$ corresponding to $\lambda$
is of finite rank, and moreover that it may be written as a uniform
limit of polynomials in $A$ with zero constant coefficient. It follows
that $P_{\lambda}\in\mathcal{J}$. This fact, combined with Theorem
\ref{thm:jor-volterra-ideal}, implies the following result.

\begin{thm}
\label{thm:jor-no-finite-rank}A uniformly closed Jordan algbra of
compact operators which doesn't contain any nonzero finite-rank operators
is reducible.
\end{thm}
For Jordan algebras, we are also interested in situations when the
reducibility of this ideal implies the reducibility of the entire
Jordan algebra. The next result establishes the Jordan analogue of
Theorem \ref{thm:lie-zero-trace-ideal}.

\begin{thm}
\label{lem:jor-tr-idl}Let $\mathcal{J}$ be a uniformly closed Jordan
algebra of compact operators. If $\mathcal{J}$ contains a nonzero
element $A$ with the property that $\operatorname{tr}(AF)=0$ for
every finite-rank operator $F$ in $\mathcal{J}$, then $\mathcal{J}$
is reducible.
\end{thm}
\begin{proof}
Let $\mathcal{F}$ denote the Jordan ideal of finite-rank elements
in $\mathcal{J}$, and let $\mathcal{I}=\{A\in\mathcal{J}:$$\operatorname{tr}(AF)=0\mbox{ for all }F\in\mathcal{F}\}$.
Using the identity\[
\operatorname{tr}(\{A,B\}F)=\operatorname{tr}(A\{B,F\})=0,\]
It is straightforward to verify that $\mathcal{I}$ is a Jordan ideal
of $\mathcal{J}$. We claim that $\mathcal{I}$ consists of quasinilpotent
elements. Indeed, suppose otherwise that for some $A\in\mathcal{I}$,
$\lambda\in\sigma(A)$ is nonzero, and let $P_{\lambda}$ be the Riesz
projection of $A$ corresponding to $\lambda$. Then $P_{\lambda}$
belongs to $\mathcal{J}$, and $\operatorname{tr}(P_{\lambda}AP_{\lambda})=n\lambda$,
where $n$ is the spectral multiplicity of $\lambda$. But\[
\operatorname{tr}(P_{\lambda}AP_{\lambda})=\operatorname{tr}(AP_{\lambda}^{2})=\operatorname{tr}(AP_{\lambda})=0\]
by hypothesis, which gives a contradiction. The result now follows
by Theorem \ref{thm:jor-volterra-ideal}.
\end{proof}
Using the structure which is present in a general Jordan algebra of
compact operators, and applying what we know about operators with
stable spectrum, we are able to prove the following result.

\begin{lem}
Let $\mathcal{J}$ be a Jordan algebra of compact operators with $T$
in $\mathcal{J}$. If $\mathcal{J}$ has $T$-stable spectrum, then
$AT$ is quasinilpotent for all $A\in\mathcal{J}$.
\end{lem}
\begin{proof}
For $A\in\mathcal{J}$, $\nu A+\nu^{2}A^{2}+...\nu^{n}A^{n}+\lambda T\in\mathcal{J}$,
for all $n\geq1$ and $\lambda,\nu\in\mathbb{C}$. Also, \[
\sigma(\nu A+\nu^{2}A^{2}+...\nu^{n}A^{n}+\lambda T)=\sigma(\nu A+\nu^{2}A^{2}+...\nu^{n}A^{n}),\]
so\[
\sigma(1+\nu A+\nu^{2}A^{2}+...\nu^{n}A^{n}+\lambda T)=\sigma(1+\nu A+\nu^{2}A^{2}+...\nu^{n}A^{n}).\]
For sufficiently small $\nu$, taking $n\to\infty$, and using the
continuity of the spectrum of compact operators gives\[
\sigma((1-\nu A)^{-1}+\lambda T)=\sigma((1-\nu A)^{-1}).\]
Hence $(1-\nu A)^{-1}$ has $T$-stable spectrum, so by Lemma \ref{lem:stable-inverse-quasinilpotent},
$(1-\nu A)T=T-\nu AT$ is quasinilpotent for sufficiently small $\nu$.
For such $\nu$, the subharmonic function $\nu\to\operatorname{log}(r(T-\nu AT))$
satisfies $\operatorname{log}(r(T-\nu AT))=-\infty$, so by H. Cartan's
Theorem, $\operatorname{log}(r(T-\nu AT))=-\infty$, i.e. $T-\nu AT$
is quasinilpotent for all $\nu\in\mathbb{C}$. But this means $T$
has $AT$-stable spectrum, so $AT$ is quasinilpotent by Lemma \ref{lem:stable-quasinilpotence}.
\end{proof}
\begin{lem}
\label{thm:jor-stable-reducible}Let $\mathcal{J}$ be a Jordan algebra
of compact operators, and let $T$ be a nonzero element of $\mathcal{J}$.
If $\mathcal{J}$ has $T$-stable spectrum, then $\mathcal{J}$ is
reducible.
\end{lem}
\begin{proof}
By the continuity of the specrum of compact operators, we may suppose
that $\mathcal{J}$ is uniformly closed. Consider the Jordan ideal
$\mathcal{F}$ of finite-rank operators in $\mathcal{J}$. By Theorem
\ref{thm:jor-no-finite-rank}, if $\mathcal{F}$ is zero then $\mathcal{J}$
is reducible; hence we may assume that $\mathcal{F}$ is nonzero.

For $A$ in $\mathcal{J}$ and $F$ in $\mathcal{F}$, $\{A,F\}$
belongs to $\mathcal{F}$, so by Lemma \ref{lem:rig-zer-tr},\[
\operatorname{tr}(\{A,F\}T)=\operatorname{tr}(\{F,T\}A)=0.\]
If $\{F,T\}\ne0$ for some $F$ in $\mathcal{F}$, then $\mathcal{J}$
is reducible by Lemma \ref{lem:jor-tr-idl}. Otherwise, if $\{F,T\}=0$
for all $F$ in $\mathcal{F}$, then\[
\operatorname{tr}(TF)=\operatorname{tr}(\frac{1}{2}\{F,T\})=0\]
for all $F$ in $\mathcal{F}$, which by Lemma \ref{lem:jor-tr-idl}
again implies the reducibility of $\mathcal{J}$.
\end{proof}
Theorem \ref{thm:jor-volterra-ideal} roughly says that a Jordan algebra
of compact operators with too many quasinilpotent operators is triangularizable.
The next result says that a Jordan algebra of compact operators with
too few quasinilpotent operators is also triangularizable. For its
proof, we will require two classical theorems. The Motzkin-Taussky
Theorem states that a linear space of finite-rank diagonalizable operators
is commutative (see for example \cite[Theorem 2.7]{kat}), and the
Kleinecke-Shirokov Theorem states if $A$ and $B$ are bounded operators
on a Banach space, and if $[A,[A,B]]]=0$, then $[A,B]$ is quasinilpotent
(see for example \cite[Theorem 5.1.3]{aup}).

\begin{lem}
\label{thm:jor-contains-qn-elems}Let $\mathcal{J}$ be a uniformly
closed Jordan algebra of compact operators. If $\mathcal{J}$ doesn't
contain any nonzero finite-rank nilpotent elements, then $\mathcal{J}$
is reducible.
\end{lem}
\begin{proof}
Suppose $\mathcal{J}$ doesn't contain any nonzero finite-rank nilpotent
elements, and let $\mathcal{F}$ be the Jordan ideal of finite-rank
elements in $\mathcal{J}$. For $F\in\mathcal{F}$, since $\mathcal{J}$
contains every polynomial in $F$ with zero constant coefficient,
in particular it contains the polynomial\[
p(F)=F\prod_{\alpha}(\alpha-F),\]
where the product is taken over all nonzero $\alpha$ in $\sigma(F)$.
Since some power of $p(F)$ is zero, $p(F)$ is nilpotent, and hence
$p(F)=0$ by hypothesis, which implies that $F$ is diagonalizable.
Hence every element in $\mathcal{F}$ is diagonalizable, so by the
Motzkin-Taussky Theorem, the elements of $\mathcal{F}$ pairwise commute.

Fix some $F\in\mathcal{F}$. For $A$ in $\mathcal{J}$, $[F,[F,A]]$
is of finite rank and belongs to $\mathcal{J}$, since\[
[F,[F,A]]=\{F,\{F,A\}\}-\{A,\{F,F\}\}.\]
From above, we therefore have that $F$ and $[F,[F,A]]$ commute,
i.e. that $[F,[F,[F,A]]]=0$, so the Kleinecke-Shirokov Theorem implies
that $[F,[F,A]]$ is nilpotent. Since, by hypothesis, $\mathcal{J}$
doesn't contain any nonzero finite-rank nilpotent elements, $[F,[F,A]]=0$,
and applying the Kleinecke-Shirokov Theorem again implies that $[F,A]$
is nilpotent. But then $[F,A]^{2}$ is also finite-rank and nilpotent,
and moreover, it belongs to $\mathcal{J}$ since

\[
[F,A]^{2}=\{A,F\}^{2}+2\{A^{2},F^{2}\}-\{A,\{A,F^{2}\}\}-\{F,\{F,A^{2}\}\}.\]
Hence $[F,A]^{2}=0$, where we again use the hypothesis that $\mathcal{J}$
doesn't contain any nonzero finite-rank nilpotent elements.

Let $B\in\mathcal{J}$ with nonzero $\beta\in\sigma(B)$, and let
$P$ be the Riesz projection of $B$ corresponding to $\beta$. Then
$P$ is of finite rank, so from above, for $A\in\mathcal{J}$,\[
0=[P,A]^{2}=(PA)^{2}-PA^{2}P-APA+(AP)^{2}.\]
Now consider $(1-P)AP+PA(1-P)=\{A,P\}-2PAP$, which belongs to $\mathcal{J}$.
We have\begin{eqnarray*}
((1-P)AP+PA(1-P))^{2} & = & (1-P)APA(1-P)+PA(1-P)AP\\
 & = & -((PA)^{2}-PA^{2}P-APA+(AP)^{2})\\
 & = & -[P,A]^{2}\\
 & = & 0,\end{eqnarray*}
which implies that the finite-rank element $(1-P)AP+PA(1-P)$ is nilpotent.
Since by hypothesis, $\mathcal{J}$ doesn't contain any nonzero finite-rank
nilpotent elements, it follows that $(1-P)AP+PA(1-P)=0$ for all $A\in\mathcal{J}$. 

We have\[
0=(1-P)AP+PA(1-P)=AP+PA-2PAP,\]
so $AP=2PAP-PA$. Multiplying on the left by $P$ then gives $PAP=2PAP-PA=PA$,
so it follows that the range of $P$ is invariant under $A$. Since
$P$ was chosen to be nontrivial, this shows that $\mathcal{J}$ is
reducible.
\end{proof}
\begin{thm}
\label{thm:jordan-sublin-tri}A Jordan algebra of compact operators
with subadditive spectrum is triangularizable.
\end{thm}
\begin{proof}
Let $\mathcal{J}$ be a Jordan algebra of compact operators with sublinear
spectrum. By the continuity of the spectrum of compact operators,
we may suppose that $\mathcal{J}$ is uniformly closed. As in Theorem
\ref{thm:lie-sub-tri}, it suffices to show the reducibility of $\mathcal{J}$.
By Theorem \ref{thm:jor-contains-qn-elems}, if $\mathcal{J}$ doesn't
contain any nonzero quasinilpotent elements, then $\mathcal{J}$ is
reducible; hence we may suppose that $\mathcal{J}$ contains a nonzero
quasinilpotent element $T$. By the hypothesis of sublinear spectrum,
$\mathcal{J}$ has $T$-stable spectrum, so the result now follows
from Theorem \ref{thm:jor-stable-reducible}.
\end{proof}
\begin{thm}
A Jordan algebra of compact operators with submultiplicative spectrum
is reducible.
\end{thm}
\begin{proof}
Let $\mathcal{J}$ be a Jordan algebra of compact operators with submultiplicative
spectrum. As in the proof of Theorem \ref{thm:jordan-sublin-tri},
we may suppose that $\mathcal{J}$ is uniformly closed. Let $\mathcal{F}$
be the ideal of finite-rank elements in $\mathcal{J}$. By Theorem
\ref{thm:jor-contains-qn-elems}, if $\mathcal{F}$ doesn't contain
a nonzero nilpotent element, then $\mathcal{J}$ is reducible; hence
we may suppose that some nonzero $T$ in $\mathcal{F}$ is nilpotent.
Then for every $F$ in $\mathcal{F}$, $\sigma(FT)=0$ by the hypothesis
of submultiplicative spectrum, so $\operatorname{tr}(FT)=0$. Hence
$\mathcal{J}$ is reducible by Theorem \ref{thm:lie-zero-trace-ideal}. 
\end{proof}
\begin{cor}
Let $\mathcal{J}$ be a uniformly closed Jordan algebra of compact
operators with the property that $\operatorname{tr}(F^{2})=0$ for
every finite-rank element $F$ in $\mathcal{J}$. Then $\mathcal{J}$
is reducible.
\end{cor}
\begin{proof}
Let $\mathcal{F}$ be the ideal of finite-rank elements in $\mathcal{J}$.
The hypothesis implies that every $F\in\mathcal{F}$ is nilpotent;
indeed, otherwise some $F\in\mathcal{F}$ would have nonzero $\alpha\in\sigma(F)$,
and the Riesz projection of $F$ corresponding to $\alpha$ would
have $\operatorname{tr}(P^{2})\ne0$, since $\sigma(P)=\{0,1\}$.
Hence the result follows by Theorem \ref{thm:jor-volterra-ideal}.
\end{proof}

\thanks{This research was partially supported by NSERC.}

\end{document}